\documentclass[11pt,twoside]{amsart}
\usepackage{amsmath}
\usepackage{amssymb}
\usepackage{mathrsfs}

\newtheorem{theo}{Theorem}

\newtheorem{lem}{Lemma}
\newtheorem{prop}{Proposition}

\theoremstyle{definition}
\newtheorem{defn}{Definition}

\theoremstyle{remark}
\newtheorem{rem}{\bf Remark\/}

\numberwithin{equation}{section}

\def\C{{\mathbb{C}}}
\def\R{{\mathbb{R}}}
\def\N{{\mathbb{N}}}
\def\1{{\mathchoice {\rm 1\mskip-4mu l} {\rm 1\mskip-4mu l}{\rm 1\mskip-4.5mu l} {\rm 1\mskip-5mu l}}}
\newcommand{\ds}{\displaystyle}
\newcommand{\w}{\wedge}
\newcommand{\p}{\partial}

\title[Existence of the directional tangent cone]{Existence of the directional tangent cone to a  positive current}

\author[H. Hawari]{Haithem Hawari}
\email{haithem.hawari@yahoo.fr}
\author[J. Hbil]{Jawhar Hbil}
\email{jawhar\_x@hotmail.fr}
\author[N. Ghiloufi]{Noureddine Ghiloufi}
\email{noureddine.ghiloufi@fsg.rnu.tn}
\address{Department of Mathematics\\ Faculty of sciences of Gab\`es \\ University of Gab\`es \\ Erriadh city 6072 Gab\`es Tunisia.}
\begin{document}
\maketitle

\begin{abstract}
    In this paper, we start by proving the existence of  the strict transform of a positive current $T$ as soon as its $j^{th}$ currents, $T_j$,  are plurisubharmonics or plurisuperharmonics. Then, with a suitable condition on $T_j$, we show the  existence of the directional tangent cone to $T$. In the particular case, when $T$ is closed, we give a second  condition independent to the previous one.

    \textbf{Existence du c\^one tangent   directionnel \`a un courant positif.}\\
    \textsc{R\'esum\'e.} Dans la premi\`ere partie de cet article, on montre l'existence du rel\'evement d'un courant positif $T$, par un éclatement de centre lisse, d\`es que ses $j^{\grave eme }$ courants $T_j$ soient  plurisousharmoniques ou plurisurharmoniques. Ensuite, avec une condition suppl\'ementaire sur les courants $T_j$, on prouve l'existence du c\^one tangent directionnel \`a $T$. Dans le cas particulier o\`u $T$ est ferm\'e, une deuxi\`eme condition ind\'ependante de la premi\`ere  sera donn\'ee.
\end{abstract}
\section{Introduction}
    We are interested to the problem of the existence of the strict transforms  of  positive currents and the problem of the existence of the directional tangent cones to positive currents. These two problems are closely related, especially in the case of  positive closed currents.

    For the second problem, it is well known that the tangent cone to the current of integration, over an analytic set, exists and coincides with the current of integration over the geometric tangent cone. This result had been proved by Thie in 1967 and King in 1971 . In the other side, Kiselman \cite {Ki} in 1988, gave an example of a positive closed  current of bidegree $(1,1)$ which   has not a tangent cone. It is therefore natural to find  a sufficient condition for its existence. Based on the construction of Kiselman which use essentially the projective mass of the current, Blel, Demailly and Mouzali \cite{Bl-De-Mo} have given two  independent conditions where each one ensure the existence of the tangent cone to a positive closed current. Recently, Haggui \cite{Ha} has shown that one of these conditions remains sufficient even for the case of positive plurisubharmonic currents. This result has been found and generalized by Ghiloufi and Dabbek \cite{Gh-Da} in the case of a positive plurisubharmonic (psh) or plurisuperharmonic (prh) current.\\

    For the first problem, Giret \cite{Gi} has given, in 1998, some estimates of the projective mass of a positive closed current which allows him to give a sufficient condition for  the existence of a strict transform of a positive current. Using Raby's and Babouche's works, \cite{Ra,Ba} on the problem of restriction of  a positive closed current on hypersurfaces, Giret gave a link between the existence of directional tangent cone and the strict transform of this current.\\

    The basic purpose  of this work is to study the existence of the directional tangent cone and the strict transform of positive psh or prh currents.\\
    In the hole of this paper, we consider $n,\ m,\ p\in\N$ such that $0<p<n$. We use $(z,t)$ to indicate an element of $\C^N:=\C^n\times\C^m$ containing 0. Let $\Omega:=\Omega_1\times\Omega_2$ be an open subset of $\C^N$ and $B$ be an open subset of $\Omega_2$.  We assume that there exists $R>0$ such that $\{z\in\C^n;\ |z|<R\}\times B$ is relatively compact in $\Omega$. For $0<r<R$ and $r_1<r_2<R$, we set: $\Bbb B_n(r)=\{z\in\Omega_1;\ |z|<r\}$, $\Bbb B_n(r_1,r_2)=\{z\in\Omega_1;\ r_1\leq|z|<r_2\}$ and $\Bbb B_m(r)=\{t\in\Omega_2;\ |t|<r\}$.\\
    To simplify the notation, we set $\beta_z:=dd^c|z|^2$, $\beta_t:=dd^c|t|^2$ and
    $\alpha_z:=dd^c \log|z|^2$.\\

    Let $T$ be a positive current of bidegree $(p,p)$ on $\Omega$. The  directional Lelong number of $T$ with respect to $B$ is defined, when it exists, as the limit, $\nu(T,B):=\lim_{r\to 0^+}\nu(T,B,r)$ where $\nu(T,B,.)$ is the function defined by
        $$\nu(T,B,r):=\frac{1}{r^{2(n-p)}}\int_{\Bbb B_n(r)\times B} T\w\beta_z^{n-p}\w\beta_t^m.$$
    In the particular case $m=0$, we omit $B$ in previous notation and to make a distinction, we note $\nu_T(r),\ \nu_T(0)$ to indication respectively classical projective mass  and Lelong number of $T$ at 0.\\
The following lemma will be useful:
\begin{lem}\label{lem1}( Lelong-Jensen formula) (See \cite{De})
            Let  $S$ be a positive psh or prh  current of bidimension $(q,q)$ on the ball  $\Bbb B_n(R)$ of  $\C^n$ with $0<q<n$. Then for any $0<r_1<r_2<R$,
            $$\begin{array}{l}
                    \nu_S(r_2)-\nu_S(r_1)\\
                    := \ds \frac1{r_2^{2q}}\int_{\Bbb B_n(r_2)} S\w\beta_z^q -\frac1{r_1^{2q}} \int_{\Bbb B_n(r_1)}S\w\beta_z^q\\
                    = \ds\int_{\Bbb B_n(r_1,r_2)}S\w\alpha_z^q +\ds \int_{r_1}^{r_2}\left(\frac1{s^{2q}} -\frac1{r_2^{2q}} \right)sds\int_{\Bbb B_n(s)}dd^cS\w\beta_z^{q-1}\\
                    \hfill \ds +\left(\frac1{r_1^{2q}}-\frac1{r_2^{2q}}\right)\int_0^{r_1}sds\int_{\Bbb B_n(s)} dd^cS\w \beta_z^{q-1} .
            \end{array}$$
        \end{lem}
According to Lemma \ref{lem1}, if $S$ is a positive psh current then $\nu_S$  is a non-negative increasing  function on $]0,R[$,  therefore the Lelong number $\nu_S(0):=\lim_{r\to0^+} \nu_S(r)$ of $S$ at $0$ exists.\\
For a positive prh current one has the following result:
\begin{prop}\label{prop1}(See \cite{Gh})
            Let  $S$ be a positive prh current of  bidimension $(q,q)$  on the ball $\Bbb B_n(R)$ of $\C^n$, $0<q<n$. We assume that $S$ satisfies the condition $(C_0)$ given by
            $$(C_0):\qquad\int_0^{r_0}\frac{\nu_{dd^cS}(s)}sds>-\infty$$ for $0<r_0< R$. Then, the Lelong number $\nu_S(0)$ of $S$ at $0$ exists.
        \end{prop}
\begin{proof}
            For $0<r<R$, we set
            $$\Lambda_S(r)=\nu_S(r)+\int_0^r\left(\frac{s^{2q}}{r^{2q}}-1\right)\frac{\nu_{dd^cS}(s)}{s}ds.$$
           The function $\Lambda_S$ is well defined and non-negative on $]0,R[$. In addition, by the Lelong-Jensen formula, it's easy to show that this function is increasing on $]0,R[$, hence the existence of the limit $\ell:=\lim_{r\to0^+} \Lambda_S(r)$. Condition $(C_0)$ and the fact that $(s^q/r^q-1)$ is uniformly bounded give $$\ds\lim_{r\to0^+}\int_0^r\left(\frac{s^{2q}}{r^{2q}}-1\right)\frac{\nu_{dd^cS}(s)}{s}ds=0.$$
            So, $\ell=\lim_{r\to0^+}\Lambda_S(r)=\lim_{r\to0^+}\nu_S(r)=\nu_S(0)$.
        \end{proof}

    We end this part by recalling the Demailly's Inequality  which will be useful in the proofs of various results in this paper:\\
    \emph{If $$S=2^{-q}i^{q^2}\sum_{|I|=|J|=q}S_{I,J}dw_I\w d\overline{w}_J$$ is a positive $(q,q)-$current, then for all $(\lambda_1,...,\lambda_n)\in]0,+\infty[^n$ we have
    $$\lambda_I\lambda_J|S_{I,J}|\leq 2^q\sum_{M\in\mathscr M_{I,J}}\lambda_MS_{M,M}$$
    where $\lambda_I=\lambda_{i_1}...\lambda_{i_q}$ when $I=(i_1,...,i_q)$ and the sum is taken over the set of $q-$index  $\mathscr M_{I,J}=\{M;\ |M|=q,\ I\cap J\subset M\subset I\cup J\}$.}\\

\section{Strict transform of a positive current}
    Let $\C^N[\{0\}\times\C^m]:=\{(z,\mathfrak{g},t)\in\C^n\times\mathbb{P}^{n-1}\times\C^m;\ z\in \mathfrak{g}\}$ be  the blow-up  of $\C^N=\C^n\times\C^m$ with center $L:=\{0\}\times\C^m$ and   $\pi$ : $\C^N[\{0\}\times\C^m]\longrightarrow\C^N$  be the canonical projection defined by $\pi(z,\mathfrak{g},t)=(z,t)$. Let $T$  be a positive current on $\C^N$. We try to find a positive current $\widehat{T}$ on $\C^N[\{0\}\times\C^m]$ such that $\pi_\star \widehat{T}=T$; a such current, if it exists, will be called the strict transform of $T$. If $m=0$, it was shown (See \cite{Gh-Da}) that this strict transform exists in the case of positive psh  currents or positive prh currents satisfying Condition $(C_0)$. That's why we will consider the case  $m\not=0$. In this case Kiselman \cite{Ki} had shown that this strict transform may not exists. We try to find some sufficient conditions on $T$  to ensure the existence of its strict transform. We define the current $T_j$ as $T_j:=\int_{\Bbb B_m(\sigma)} T\w \beta_t^{m-j}$ for any integer $0\leq j\leq m$.   The current $T_j$ is positive  of bidimension $(n-p+j,n-p+j)$ on $\C^n$ where $(p,p)$ is the bidegree of $T$. We assume that the currents $T_j$ are psh or prh satisfying Condition $(C_0)$, so the Lelong number $\nu_{T_j}(0)$ of $T_j$ at 0 exists; this number will be called the $j^{th}$ Lelong number  of $T$ at 0.\\
     \begin{theo}
        Let $T$ be as above. Assume that for $j\in J_1$ (resp. $j\in J_2$), the current  $T_j:=\int_{\Bbb B_m(\sigma)}T\w\beta_t^j$ is plurisubharmonic (resp. plurisuperharmonic satisfying Condition $(C_0)$) for every $\sigma>0$ where $J_1\cup J_2=\{0,...,m\}$. Then the strict transform of  $T$ exists. furthermore, there exists a constant $ c> 0 $ such that for every $0<r<R$, one has
        $$\begin{array}{l}
            ||\pi^\star(T_{|\Omega\smallsetminus L})||(\pi^{-1}[(\Bbb B_n(r)\smallsetminus \{0\})\times\Bbb B_m(\sigma)])\\
            \leq \ds c\sum_{j=0}^{m}\left(|\nu_{T_j}(r)-\nu_{T_j}(0)|+C_1(r)\nu_{T_j}(r)\right)-C_2(r) \sum_{j\in J_2}\ds\int_0^r\frac{\nu_{dd^cT_{m-j}}(s)}sds.
        \end{array}$$
        where     $C_1(r)= \sum_{k=1}^{N-p}r^{2k}$       and       $C_2(r)=\sum_{k=0}^{N-p} r^{2k}$
      \end{theo}
      This Theorem generalizes a result of Giret where he considered the case $J_2=\emptyset$ (all $T_j$ are psh).
    \begin{proof}
        Let $\omega$ be the K\"ahler form of $\C^N[\{0\}\times\C^m]$.
      We have:
        $$\begin{array}{lcl}
         \ds(\pi_\star\omega)^{N-p}&=&\left(\beta_t+\beta_z+\alpha_z\right)^{N-p}\\
         &=&\ds\sum_{k=0}^{N-p}C_{N-p}^k (\beta_t+\beta_z)^k\w\alpha_z^{N-p-k}\\
         &=&\ds\sum_{k=0}^{N-p}C_{N-p}^k \left(\sum _{j=0}^{\min(m,k)}C_{k}^j\beta_t^j\w\beta_z^{k-j}\right)\w\alpha_z^{N-p-k}.\\
        \end{array}
        $$

        It is therefore sufficient to control the integrals:
        $$\ds\int_{\Bbb B_n(\varepsilon,r)\times\Bbb B_m(\sigma)}T\w \alpha_z^{N-p-k}\w\beta_z^{k-j}\w\beta_t^{j}$$
        for any integers $0\leq k\leq N-p,\ 0\leq j\leq \min(k,m)$  and $0<\varepsilon<r$. \\

        For $0\leq k\leq N-p$ and $0\leq j\leq \min(k,m),$ the Lelong-Jensen formula applied to $T_{m-j}\w\beta_z^{k-j}$ gives:

        \begin{equation}\label{eq2.1}
            \begin{array}{ll}
                &\ds\int_{\Bbb B_n(\epsilon,r)\times\Bbb B_m(\sigma)}T\w \alpha_z^{N-p-k}\w\beta_z^{k-j}\w\beta_t^{j}\\
                =&\ds\frac1{r^{2(N-p-k)}}\int_{\Bbb B_n(r)\times\Bbb B_m(\sigma)} T\w\beta_z^{N-p-j}\w\beta_t^{j}\\
                &-\ds\frac1{\epsilon^{2(N-p-k)}}\int_{\Bbb B_n(\epsilon)\times\Bbb B_m(\sigma)} T\w\beta_z^{N-p-j}\w\beta_t^{j}\\
                &-\ds\int_0^r\left(\frac{1}{s^{2(j-k)}}-\frac{s^{2(N-p-j)}}{r^{2(N-p-k)}}\right)\frac{\nu_{dd^cT_{m-j}}(s)}sds\\
                &+\ds\int_0^\epsilon\left(\frac{1}{s^{2(j-k)}} -\frac{s^{2(N-p-j)}}{\epsilon^{2(N-p-k)}}\right)\frac{\nu_{dd^cT_{m-j}}(s)}sds\\
                =&\ds r^{2(k-j)}\nu_{T_{m-j}}(r)-\ds\int_0^rs^{2(k-j)} \left(1-\frac{s^{2(N-p-k)}}{r^{2(N-p-k)}}\right)\frac{\nu_{dd^cT_{m-j}}(s)}sds\\
                & -\epsilon^{2(k-j)}\nu_{T_{m-j}}(\epsilon)+\ds\int_0^\epsilon s^{2(k-j)}\left(1-\frac{s^{2(N-p-k)}}{\epsilon^{2(N-p-k)}}\right)\frac{\nu_{dd^cT_{m-j}}(s)}sds.
            \end{array}
        \end{equation}

        If $j\in J_1$, Equality \ref{eq2.1} gives
        $$\begin{array}{ll}
            &\ds\int_{\Bbb B_n(\epsilon,r)\times\Bbb B_m(\sigma)}T\w \alpha_z^{N-p-k}\w\beta_z^{k-j}\w\beta_t^{j}\\
            \leq&\ds r^{2(k-j)}\nu_{T_{m-j}}(r)-\epsilon^{2(k-j)}\nu_{T_{m-j}}(\epsilon)\\
            & +\ds\int_0^\epsilon s^{2(k-j)}\left(1-\frac{s^{2(N-p-k)}}{\epsilon^{2(N-p-k)}}\right)\frac{\nu_{dd^cT_{m-j}}(s)}sds.
        \end{array}$$
        If $j\in J_2$, Equality \ref{eq2.1} gives
        $$\begin{array}{ll}
            &\ds\int_{\Bbb B_n(\epsilon,r)\times\Bbb B_m(\sigma)}T\w \alpha_z^{N-p-k}\w\beta_z^{k-j}\w\beta_t^{j}\\
            \leq&\ds r^{2(k-j)}\nu_{T_{m-j}}(r)-r^{2(k-j)}\ds\int_0^r\frac{\nu_{dd^cT_{m-j}}(s)}sds\\
            & -\epsilon^{2(k-j)}\nu_{T_{m-j}}(\epsilon)+\ds\int_0^\epsilon s^{2(k-j)}\left(1-\frac{s^{2(N-p-k)}}{\epsilon^{2(N-p-k)}}\right)\frac{\nu_{dd^cT_{m-j}}(s)}sds.
        \end{array}$$

       Taking the  limit when $\epsilon$ goes to $0$,  we deduce that the current $\pi^{\star}(T_{|\Omega{\smallsetminus L}})$ has a locally finite mass in a neighborhood of points of $E=\pi^{-1}(L)$. More precisely, there exists a constant $c$ such that
        $$\begin{array}{ll}
            &||\pi^\star(T_{|\Omega\smallsetminus L})||(\pi^{-1}\{(\Bbb B_n(r)\smallsetminus \{0\})\times\Bbb B_m(\sigma)\})\\
            \leq& c \ds \sum_{j=0}^{m}\left(|\nu_{T_{m-j}}(r)-\nu_{T_{m-j}}(0)|+C_1(r)\nu_{T_{m-j}}(r)\right)\\
            & \ds\hfill-\sum_{j\in J_2}C_2(r) \ds\int_0^r\frac{\nu_{dd^cT_{m-j}}(s)}sds.
        \end{array}$$
        Then the trivial extension $\widehat{T}$ of $\pi^\star(T|_{\Omega\smallsetminus L})$ by zero over $\pi^{-1}(L)$ exists. The current $\pi_\star\widehat{T}-T$ is $\C-$flat of bidimension $(n+m-p,n+m-p)$ on $\C^N$ supported by $\{0\}\times\C^m$, hence it vanishes by the support theorem.
    \end{proof}
    \begin{rem} Let $T$ be a positive current with support in the stripe $\Omega_1\times \Bbb B_m(\sigma_0)$.
        \begin{enumerate}
          \item If $T$ is  plurisubharmonic then $T$ has a positive strict transform. 
          \item If $T$ is  plurisuperharmonic and $T_j=\int_{\Bbb B_m(\sigma_0)}T\w\beta_t^{m-j}$ satisfy Condition $(C_0)$, $0\leq j\leq m$, then $T$ has a positive strict transform.
        \end{enumerate}        
    \end{rem}

  \section{Directional tangent cone to a positive current}
 In this section we denote by  $ T_j:=\int_B T\w \beta_t^{m-j}$ where $B$ is an open set of $\C^m$.
 \begin{defn}
     Let  $T$ be a positive current on an open set $\Omega$ of $\C^N=\C^n\times\C^m$. The directional tangent cone to  $T$ is the weak limit, if it exists, of the family  $\left( {h_a}^\star T\right)_{a\in\C}$ as $a\rightarrow 0$ where  $h_a$ is the directional dilatation  on $\C^N$ defined by $h_a(z,t)=(az,t)$.
 \end{defn}
 The main result of this paper is  the following:
     \begin{theo}\label{th3}
         Let $T$ be a positive current of bidegree $(p,p)$ on $\Omega $. We assume that for any integer $0\leq j\leq \min(m,p)$, the current $T_j$ is psh (resp. prh satisfying Condition $(C_0)$) and
         $$\ds \int_0^{r_0}\frac{|\nu_{T_j}(r)-\nu_{T_j}(0)|}{r}dr < +\infty.$$
         Then the directional tangent cone to $T$ exists.
      \end{theo}

  We start by giving some results useful for the proof of this theorem.
       \begin{rem}
            Let $T$ be  as above and $B$ an open set of $\C^m$, then for all integer $0\leq j\leq m$, the current $T_j$
          is positive, of bidimension  $(n-p+j,n-p+j)$ on  $\C^n$ and one has:
            $$h_a^\star T_j=(h_a^\star T)_j.$$
        \end{rem}

        \textbf{In  fact}, we can assume, without loss of generality, that $T$ is $\mathcal C^\infty$. Then
      $$h_a^\star\int_BT\w \beta_t^{m-j}=h_a^\star {p_1}_\star\left({T \w \beta_t^{m-j}}|_{\C^n\times B}\right)$$
      where
      $$\begin{array}{llcl}
        {p_1} :&\C^n\times \C^m & \longrightarrow  & \C^n\times\C^m \\
       & (z,t) & \longmapsto & (z,0).
      \end{array}$$
      One has
      $$\begin{array}{lcl}
      \ds h_a^\star {p_1}_\star\left({T \w \beta_t^{m-j}}|_{\C^n\times B}\right)&=&\ds{(h_{\frac{1}{a}}})_\star {p_1}_\star\left({T \w \beta_t^{m-j}}|_{\C^n\times B}\right)\\
      &=&\ds\left( {p_1}\circ h_{\frac{1}{a}}\right)_\star \left({T \w \beta_t^{m-j}}|_{\C^n\times B}\right)\\
      &=&\ds\left(h_{\frac{1}{a}}\circ {p_1}\right)_\star \left({T \w \beta_t^{m-j}}|_{\C^n\times B}\right)\\
      &=&\ds {p_1}_\star\left( h_a^\star {T\w \beta_t^{m-j}}|_{\C^n\times B}\right)\\
      &=&\ds \int_B h_a^\star T\w \beta_t^{m-j}.
      \end{array}$$

   \hfill $\square$
          \begin{lem}
              Let $T$ be a positive current of bidimension $(p,p)$ on an open set  $\Omega $ of $\C^n\times \C^m$. We assume that the current $T_j$ is psh for $j\in J_1$ and  prh satisfying  Condition $(C_0)$ for $j \in J_2$ where $J_1\cup J_2= \{0,...,\min(m,p)\}$. Then there exists $c>0$ such that
              \begin{equation}\label{eq3.1}
                \begin{array}{l}
                    \ds\int_{\Bbb B_n(r)\times B}h_a^\star T\w (\beta_z+\beta_t)^{N-p}\\
                    \ds\leq C r^{2(n-p)}\left(\sum_{j=0}^{\min(m,p)}\nu_{T_j}(r_0) +\sum_{j\in J_2}\int_0^{r_0}\left(\frac{s^{2(n-p+j-1)}}{r_0^{2(n-p+j-1)}}-1\right)\frac{\nu_{dd^cT_j}(s)}{s}ds.\right)
                \end{array}
             \end{equation}
           for $|a|r\leq r_0$.
        \end{lem}
     \begin{proof} For $r>0$, we have
        $$\ds\int_{\Bbb B_n(r)\times B}h_a^\star T\w (\beta_z+\beta_t)^{N-p}=
                \ds \sum_{j=0}^{\min(m,p)}C_{N-p}^j\ds\int_{\Bbb B_n(r)\times B}h_a^\star T\w \beta_z^{n-p+j}\w \beta_t^{m-j}.$$
        Furthermore, if $j\in J_1$  then $\nu_{h_a^\star T_j}(r)=\nu_{T_j} (|a|r)\leq\nu_{T_j} (r_0)$ for $r\leq \frac{r_0}{|a|}\leq \min(1,\frac{r_0}{|a|})$. It follows that:
        $$\int_{\Bbb B_n(r)\times B}h_a^\star T\w \beta_z^{n-p+j}\w\beta_t^{m-j}\leq  r^{2(n-p+j)}\nu_{T_j}(r_0)\leq r^{2(n-p)}\nu_{T_j}(r_0).$$
        If $j\in J_2$ then using the proof of Proposition \ref{prop1}, we have $$\ds\frac{1}{r^{2(n-p+j)}}\int_{\Bbb B_n(r)\times B}h_a^\star T\w \beta_z^{n-p+j}\w \beta_t^{m-j}\leq\Lambda_{h_a^\star T_j}(r)\leq\Lambda_{T_j}(r_0).$$ It follows that  $$\ds \int_{\Bbb B_n(r)\times B}h_a^\star T \w \beta_z^{n-p+j}\w \beta_t^{m-j}\leq r^{2(n-p)}\Lambda_{T_j}(r_0).$$ The result follows by summing on $j$ from  0 to $\min(m,p)$.
    \end{proof}

According to the previous lemma the mass $h_a^\star T$ converges in a neighborhood of $(0,t)$, then the sequence $h_a^\star T$ converges on $\C^N$ if and only if it converges on a neighborhood of $(z_0,t)$  where $z_0\in\C^n\smallsetminus\{0\}$.\\
Using a dilatation and changement of coordinates if necessary, we can assume that $(z_0,t)=(0,...,z_n^0,t)$ where
$\frac{1}{2}<|z_n^0|<1$. We use the projective coordinates:
$$w_1=\frac{z_1}{z_n},...,w_n=z_n$$
$$T=\sum _{I,J,K,L} T_{IJ,KL}dw_I\w d\overline{w}_J\w dt_K\w d\overline{t}_L$$
$h_a$ will be written as: $h_a:(w',w_n,t)\mapsto(w',aw_n,t).$
We check that the  coefficients $T^a_{IJ,KL}$ of $h_{a}^\star T$ are given by:
$$ T^a_{IJ,KL}=\left\{\begin{array}{lcl}
                   T_{IJ,KL}(w',w_n,t) & if & n\not\in I \ and\ n\not \in J\\
                   a\;T_{IJ,KL}(w',w_n,t) & if &  n\in I \ and\  n\not \in J\\
                   \overline{a}\;T_{IJ,KL}(w',w_n,t) & if & n\not\in I \ and \ n\in J \\
                   |a|^2\;T_{IJ,KL} (w',w_n,t)& if & n\in I \ and \ n \in J
                 \end{array}\right.$$

\begin{lem} \label{lem4}
    Let $U$ be a neighborhood of $z_0 \in \C^n$ such that $U\times B \subset \Bbb B_n(\frac{1}{2},1)\times B$. If we note, for every integer $0\leq j\leq m$, by
    $$\gamma_{T_j}(r)=\nu_{T_j}(r)-\nu_{T_j}(r/2)\quad and \quad \gamma_{dd^cT_j}(r)=\nu_{dd^cT_j}(r)-\nu_{dd^cT_j}(r/2)$$
   then for  $|a|<r_0$, $r_0<R$, there exists $C_1,C_2,C_3>0$ such that the measure $T^a_{IJ,KL}$ satisfies the following estimates:
    $$\ds \int_{U\times B}|T^a_{IJ,KL}|\leq\left\{\begin{array}{lcl}
                    C_1 &  &\\
                   C_2 \ds\sum_{j=0}^m\gamma_{T_j}(|a|)+\gamma_{dd^cT_j}(|a|) & if  & n\in I\  and\  n\in J\\
                   C_3 \ds\left(\sum_{j=0}^m\gamma_{T_j}(|a|)+\gamma_{dd^cT_j}(|a|)\right)^{\frac12} & if  & n\in I\  or\  n\in J                 \end{array}\right.$$
\end{lem}

\begin{proof}\
 \begin{itemize}
          \item The set $\overline{U}$ is compact and the positive form $\beta$ is $\mathcal{C}^{\infty}$, so one has: $\beta\geq C_4dd^c|w|^2$ on $U$. The inequality (\ref{eq3.1}), with $r=1$, implies: $\ds \int_{U\times B}T_{M,M}^a\leq C_5$ . The Demailly's inequality with the choice \\ $\lambda_1=...=\lambda_N=1$, gives
              $$\int_{U\times B}|T_{IJ,KL}^a|\leq C_6\sum_{M\in\mathscr M_{IJ,KL}}\int_{U\times B}T_{M,M}^a\leq C_1.$$
              The first estimate of the lemma is proved.
          \item To proof the second estimate, we observe that  $\alpha\geq C_7 \beta'$ on $U$ with $\beta'=dd^c|w'|^2$. In fact, $\ds\alpha=dd^c\log(1+|w'|^2)\geq \frac1{(1+|w'|^2)^2}\beta'\geq \frac14\beta'$ on $U$. It follows that:
              $$\begin{array}{lcl}
                  \ds\int_{U\times B}\sum_{n\in M}T_{M,M}^a & = & \ds \int_{U\times B} h_a^\star T\w (dd^c|w'|^2+\beta_t)^{N-p} \\
                  & \leq &\ds C_8\sum_{j=0}^{m}\int_{U\times B} h_a^\star T\w \beta_t^{m-j}\w \alpha^{n-p+j}\\
                  &\leq& \ds C_8\sum_{j=0}^{m}\int_{\Bbb B_n(1/2,1)\times B} h_a^\star T\w \beta_t^{m-j}\w \alpha^{n-p+j}.
                \end{array}$$
              Using Lelong-Jensen formula for $T_j=\int_B T\w \beta_t^{m-j}$, $r_2=1$ and $r_1=1/2$, one has:

              $$\begin{array}{l}
                   \ds\int_{\Bbb B_n(1/2,1)\times B} h_a^\star T\w \beta_t^{m-j}\w \alpha^{n-p+j}\\
                  \leq  \ds \nu_{T_j}(|a|)-\nu_{T_j}(|a|/2)-\int_{\frac12}^1\left(\frac1{t^{2p}} -1 \right)t^{2p-1}\nu_{dd^c(h_a^\star T_j)}(t)dt\\
                   \hfill\ds-\left(\frac1{2^{2p}}-1\right) \int_0^{\frac12} t^{2p-1}\nu_{dd^c(h_a^\star T_j)}(t)dt\\
                  \ds \leq\nu_{T_j}(|a|)-\nu_{T_j}(|a|/2)-\int_{\frac12}^1\frac{\nu_{dd^cT_j}(|a|t)}t dt
                  \ds +\int_0^1t^{2p-1}\nu_{dd^cT_j}(|a|t)dt\\
                   \ds \leq(\nu_{T_j}(|a|)-\nu_{T_j}(|a|/2))+C_9(\nu_{dd^cT_j}(|a|)-\nu_{dd^cT_j}(|a|/2))\\
                   \leq  \ds C_8\gamma_{T_j}(|a|)+C_9\gamma_{dd^cT_j}(|a|).
              \end{array}$$

              Hence $$\ds\int_{U\times B}\sum_{n\in M}T_{M,M}^a \leq \sum_{j=0}^{m} C_8\gamma_{T_j}(|a|)+C_9\gamma_{dd^cT_j}(|a|).$$
              The second estimate is proved  for $n\in M$.\\
              In the general case, $I,J\ni n$, using the Demailly's inequality for \\$\lambda_1=...=\lambda_N=1$, we obtain:
              $$\int_{U\times B}|T_{IJ,KL}^a|\leq C_{10}\sum_{M\in\mathscr M_{IJ,KL}}\int_{U\times B}T_{M,M}^a\leq C_2\sum_{j=0}^{m}(\gamma_{T_j}(|a|)+\gamma_{dd^cT_j}(|a|))$$
              Hence we get The second estimate.
          \item For the third estimation, it suffices to assume that $n\in I$ and $n\not\in J$. Using the Demailly's inequality again with $\lambda_k=1$ for every $k\neq n$ and $\lambda_n>0$, one has:
              $$\begin{array}{l}
                  \ds\lambda_n\int_{U\times B} |T_{IJ,KL}^a|\\
                  \leq \ds  C_{11}\int_{U\times B}\left(\sum_{n\not\in M\in\mathscr M_{IJ,KL}}T_{M,M}^a +\lambda_n^2\sum_{n\in M\in\mathscr M_{IJ,KL}}T_{M,M}^a \right) \\
                  \leq\ds C_{12}+ C_{13}\lambda_n^2\sum_{j=0}^{m}(\gamma_{T_j}(|a|)+\gamma_{dd^cT_j}(|a|)).
                \end{array}$$
              The third estimate can be deduced by choosing $$\lambda_n=\frac1{\ds\left(\sum_{j=0}^{m}\gamma_{T_j}(|a|) +\gamma_{dd^cT_j}(|a|)\right)^{\frac12}}.$$
        \end{itemize}
\end{proof}

    \begin{lem} (See \cite{Bl-De-Mo}) \label{lem5}
        Let $f$ be a  $\mathcal{C}^2$ function defined on the pointed disc $\mathbb D^*(0,r_0)\subset\C$. We assume that $f$ is bounded and there exists a measurable function $u:]0,r_0]\rightarrow \R^+$ where $\int_0^{r_0}r|\log r|u(r)dr<+\infty$ and $ |\Delta f (a)|\leq u(|a|)$.
        Then $f(a)$ has a limit when $a$ goes to $0$.
    \end{lem}

\subsection{Proof of Theorem \ref{th3}}

The proof of the main theorem will be done in the case of positive psh currents, the case of positive prh currents is similar with some simple modifications, that's why we assume that $T$ is positive psh. The lemma \ref{lem4} shows that the mass of $T^a_{IJ,KL}$ goes to  $0$ when $n\in I$ or  $n\in J.$\\
For $n\not\in I=\{i_1,..,i_q\}$ and  $n\not \in J=\{j_1,..,j_{q'}\},$ let $\varphi\in \mathcal{D}(U\times B)$. We set \\
$$\begin{array}{lcl}
    f_{IJ,KL}(a)&=&\ds\int_{U\times B}T^a_{IJ,KL}(w,t)\varphi(w,t)d\tau(w,t)\\
    &=&\ds\int_{U\times B}T_{IJ,KL}(w',aw_n,t)\varphi(w,t)d\tau(w,t).
\end{array}$$
The function $f_{IJ,KL}$ is $\mathcal C^\infty$ on $\mathbb D^\star(0,R):=\{a\in \C;0<|a|<R\}$ and bounded in a neighborhood of $0$. We have:
$$\frac{\p^2f_{IJ,KL}}{\p a\p \overline{a}}(a)=\ds\int_{U\times B}|w_n|^2\frac{\p^2T_{IJ,KL}}{\p w_n\p \overline{w}_n}(w',w_n,t)\varphi(w,t)d\tau(w,t).$$
We observe that the coefficients  $dw_{I\cup \{n\}}\w d\overline{w}_{J\cup \{n\}}\w dt_K\w d\overline{t}_J$ in the expression of  $dd^c T$ is
$$\begin{array}{ll}
            &\ds (dd^cT)_{I\cup\{n\}J\cup\{n\},KL}\\
            =&\ds(-1)^{q'}\frac{\p^2T_{IJ,KL}}{\p w_n \p \overline{w}_n} + \sum_{k=1}^q \sum_{s=1}^{q'}(-1)^{q+s+k}\frac{\p^2T_{I(k)J(s),KL}}{\p w_{i_k} \p \overline{w}_{j_s}}\\
            &\ds+\sum_{s=1}^{q'}(-1)^{s-1}\frac{\p^2T_{IJ(s),KL}}{\p w_{n} \p \overline{w}_{j_s}} + \sum_{k=1}^q(-1)^{k-1}\frac{\p^2T_{I(k)J,KL}}{\p w_{i_k} \p \overline{w}_{n}}\\
            &+\ds\sum_{e=1}^{p-q} \sum_{e'=1}^{p-q'}(-1)^{q'+q+e'}\frac{\p^2T_{I\cup\{n\}J\cup\{n\},K_e L_{e'}}}{\p t_{i_e} \p \overline{t}_{j_{e'}}}\\
            &+\ds \sum_{e'=1}^{p-q'}(-1)^{e'-q'}\frac{\p^2T_{I J\cup\{n\},KL_f}}{\p w_{n} \p \overline{t}_{j_{e'}}}+\ds \sum_{e=1}^{p-q}(-1)^{e}\frac{\p^2T_{I\cup\{n\} J,K_e L}}{\p t_{i_e} \p \overline{w}_{n}}.
        \end{array}$$
with $I(k)=I\smallsetminus\{i_k\}\cup \{n\}$, $J(s)=J\smallsetminus\{j_s\}\cup \{n\}$, $K_e=K\smallsetminus\{i_e\}$ et  $L_{e'}=L\smallsetminus\{j_{e'}\}$. The previous equality gives:

 $$\begin{array}{ll}
            &\ds\frac{\p^2f_{IJ,KL}}{\p a\p\overline{ a}}(a)\\
            =&\ds(-1)^{q'}\int_{U\times B}\frac{|w_n|^2}{|a|^2}(dd^cT)^a_{I\cup\{n\}J\cup\{n\},KL}\varphi(w,t)d\tau(w,t)\\
            &+\ds  \sum_{k=1}^q \sum_{s=1}^{q'}(-1)^{+q+q'+1s+k}\ds\int_{U\times B}\frac{|w_n|^2}{|a|^2}T^a_{I(k)J(s),KL}\frac{\p ^2\varphi}{\p w_{i_k}\p \overline{w}_{j_s}}d\tau(w,t)\\
            &+\ds \sum_{k=1}^{q}(-1)^{k+q'}\ds\int_{U\times B}\frac{|w_n|^2}{a}T^a_{I(k)J,KL}\frac{\p ^2\varphi}{\p w_{i_k}\p \overline{w}_{n}}d\tau(w,t)\\
            &+\ds \sum_{s=1}^{q'}(-1)^{s+q'}\ds\int_{U\times B}\frac{|w_n|^2}{\overline{a}}T^a_{IJ(s),KL}\frac{\p ^2\varphi}{\p w_{n}\p \overline{w}_{j_s}}d\tau(w,t)\\
            &+\ds  \sum_{e=1}^{p-q} \sum_{e'=1}^{p-q'}(-1)^{q+e'+1}\ds\int_{U\times B}\frac{|w_n|^2}{|a|^2}T^a_{I\cup\{n\}J\cup\{n\},K_eL_{e'}}\frac{\p ^2\varphi}{\p t_{i_e}\p \overline{t}_{j_{e'}}}d\tau(w,t)\\
            &+\ds \sum_{e'=1}^{p-q'}(-1)^{e'+1}\ds\int_{U\times B}\frac{|w_n|^2}{a}T^a_{IJ\cup\{n\},KL_{e'}}\frac{\p ^2\varphi}{\p w_{n}\p \overline{t}_{j_{e'}}}d\tau(w,t)\\
            &+\ds \sum_{e=1}^{p-q}(-1)^{e+q'+1}\ds\int_{U\times B}\frac{|w_n|^2}{\overline{a}}T^a_{I\cup\{n\}J,K_eL}\frac{\p ^2\varphi}{\p t_{i_{e}}\p \overline{w}_{n}}d\tau(w,t).
        \end{array}$$
The lemma \ref{lem4} gives:
$$\begin{array}{l}
            \ds\left|\frac{\partial ^2 f_{IJ,KL}}{\partial a\partial \overline{a}}(a)\right|\\
            \leq \ds C_1\sum_{j=0}^m\frac{\gamma_{dd^cT_j}(|a|)}{|a|^2}+C_2\sum_{j=0}^m\frac{\gamma_{T_j}(|a|) +\gamma_{dd^cT_j}(|a|)}{|a|^2}\\
            \ds\hfill+C_3 \sum_{j=0}^m\frac{\sqrt{\gamma_{T_j}(|a|)+\gamma_{dd^cT_j}(|a|)}}{|a|}\\
            \leq  \ds C \left(\sum_{j=0}^m\frac{\gamma_{T_j}(|a|)+\gamma_{dd^cT_j}(|a|)}{|a|^2} +\sum_{j=0}^m\frac{\sqrt{\gamma_{T_j}(|a|)+\gamma_{dd^cT_j}(|a|)}}{|a|}\right)\\
            =C \psi(|a|).
        \end{array}$$

         According to lemma \ref{lem5}, the function $f_{IJ,KL}(a)$ has a limit at $0$ if $\psi$ satisfies
        $$\int_0^{r_0}r|\log r|\psi(r)dr<+\infty.$$

        It's easy to check that there is equivalence between $$\int_0^{r_0}\sum_{j=0}^m\frac{\gamma_{T_j}(r)+\gamma_{dd^cT_j}(r)}r|\log r|dr<+\infty $$
       and $$\int_0^{r_0}\sum_{j=0}^m\frac{\nu_{T_j}(r)-\nu_{T_j}(0)}rdr<+\infty\quad \& \quad \int_0^{r_0}\sum_{j=0}^m\frac{\nu_{dd^cT_j}(r)}rdr<+\infty.$$
        These conditions are exactly the assumptions of the main theorem. So we conclude that
        \begin{equation}\label{eq3.2}
            \int_0^{r_0}\sum_{j=0}^m\frac{\gamma_{T_j}(r)+\gamma_{dd^cT_j}(r)}r|\log r|dr<+\infty.
        \end{equation}
       Using Cauchy-Schwarz's Inequality, Equation (\ref{eq3.2}) gives
        $$\begin{array}{l}
            \ds\int_0^{r_0}\sqrt{\gamma_{T_j}(r)+\gamma_{dd^cT_j}(r)}|\log r|dr\\
            \leq  \ds \left( \int_0^{r_0}\frac{\gamma_{T_j}(r)+\gamma_{dd^cT_j}(r)}r|\log r|dr\right)^{1/2}\times \left( \int_0^{r_0}r|\log r|dr\right)^{1/2}
            <+\infty.
        \end{array}$$
       It follows that:
        $$\int_0^{r_0}r|\log r|\psi(r)dr<+\infty$$
        This completes the proof of the main theorem.

\subsection{Case of closed currents}
    In the case of positive closed currents, one has  a second condition which ensures the existence of the tangent cone, and it is given by the following theorem, where we use the following lemma to show it:
\begin{lem}(See \cite{Bl-De-Mo})\label{lem6}
    Let  $g$ be a  $\mathcal C^1$ function defined on the pointed disc $\mathbb D^*(0,r_0)$. We assume that it exists a measurable  function $u:]0,r_0]\rightarrow \R_+$ satisfying $\int_0^{r_0}u(r)dr<+\infty$ such that: $|dg(a)|\leq u(|a|).$ Then $g(a)$ has a limit when $a$ goes to $0$.
\end{lem}

\begin{theo}\label{th4}
         Let $T$ be a positive closed current of bidegree $(p,p)$ on an open set  $\Omega $ of $\C^n\times\C^m.$ We assume that for all integer $j \in [0,m]$ the current $T_j$ is closed and
         $$\ds \int_0^{r_0}\frac{\sqrt{\nu_{T_j}(r)-\nu_{T_j}(\frac r2)}}{r}dr < +\infty.$$
         Then the directional tangent cone to $T$ exists.
\end{theo}
\begin{proof}
We set \\
$$\begin{array}{lcl}
    f_{IJ,KL}(a)&=&\ds\int_{U\times B}T^a_{IJ,KL}(w,t)\varphi(w,t)d\tau(w,t)\\
    &=&\ds\int_{U\times B}T_{IJ,KL}(w',aw_n,t)\varphi(w,t)d\tau(w,t).
\end{array}$$
By derivation under the integral sign one has:
$$\frac{\p f_{IJ,KL}}{\p a}(a)=\ds\int_{U\times B}w_n\frac{\p T_{IJ,KL}}{\p w_n}(w',aw_n,t)\varphi(w,t)d\tau(w,t).$$
$$\frac{\p f_{IJ,KL}}{\p \overline{a}}(a)=\ds\int_{U\times B}\overline{w}_n \frac{\p T_{IJ,KL}}{ \p \overline{w_n}}(w',aw_n,t)\varphi(w,t)d\tau(w,t).$$
The coefficient of $dw_{I\cup\{n\}}\w d\overline{w}_J\w dt_K\w d\overline{t}_L$ in $dT$ is given by
$$(-1)^q \frac{\p T_{IJ,KL}}{\p w_n}+\sum_{1\leq k\leq q} (-1)^{k-1}\frac{\p T_{I(k)J,KL}}{\p w_{i_k}}+  \sum_{1\leq e\leq {p-q}}(-1)^{k-1}\frac{\p T_{I\cup\{n\}J,K_eL}}{\p t_{i_e}}$$
With $ I(k)=(I\smallsetminus \{i_k\}\cup\{n\}),$ and $K_e=K\smallsetminus \{i_e\}$.\\
This coefficient vanishes because $dT=0$ and one has $$T_{IJ,KL}(w',aw_n,t)=a^{-1}T_{IJ,KL}^a(w,t).$$ It follows that
$$\begin{array}{l}
    \ds\frac{\p T_{IJ,KL}}{\p w_n}(w',aw_n,t)\\
    =\ds\frac1 a \sum_{1\leq k\leq q}(-1)^{k+q}\frac{\p T_{I(k)J,KL}^a}{\p w_{i_k}}+\frac{1}{a}\sum_{1\leq e\leq {p-q}}(-1)^{q+e-1}\frac{\p T_{I\cup\{n\}J,K_eL}^a}{\p t_{i_e}}
  \end{array}
$$
By substituting this relation in the integral one has $\frac{\p f_{IJ,KL}}{\p a}$ and a by integration by parts we obtain:
$$\begin{array}{lcl}
\ds\frac{\p f_{IJ,KL}}{\p a}&=&\ds\frac1 a\sum_{1\leq k\leq q}(-1)^{k+q-1}\ds \int_{U\times B} w_n T_{I(k)J,KL}^a\frac{\p \varphi} {\p w_{i_k}}d\tau(w,t)\\
&&+\ds\frac1 a\sum_{1\leq e\leq {p-q}}(-1)^{q+e-1}\int_{U\times B} w_n T_{I\cup\{n\}J,K_eL}^a\frac{\p \varphi} {\p t_{i_e}}d\tau(w,t)
\end{array}$$
Naturally we have:
$$\begin{array}{lcl}
\ds \frac{\p f_{IJ,KL}}{\p \overline{a}}&=&\ds\frac{1} {\overline{a}}\sum_{1\leq l\leq q}(-1)^{l+q-1}\ds \int_{U\times B} \overline{w}_n T_{IJ(l),KL}^a\frac{\p \varphi} {\p \overline{w}_{i_l}}d\tau(w,t)\\
&&+ \ds\frac{1} {\overline{a}}\sum_{1\leq e'\leq {p-q'}}(-1)^{q+e'-1}\int_{U\times B} \overline{w}_n T_{I\cup\{n\}J,KL_{e'}}^a\frac{\p \varphi} {\p \overline{t}_{j_{e'}}}d\tau(w,t)
\end{array}$$
The function $\varphi$ and its derivatives are bounded on $U\times B$. The lemma \ref{lem4} gives the following estimate
$$\left|\frac{\p f_{IJ,KL}}{\p a}\right|+\left|\frac{\p f_{IJ,KL}}{\p \overline{a}}\right|\leq C_1\frac{1}{|a|}\sqrt{\sum_{j=0}^{m}\gamma_{T_j}(|a|)},$$
  According to lemma \ref{lem6}, the function $ f_{IJ,KL}$ has a limit at $0$ and the directional tangent cone to $T$ exists.
\end{proof}

\begin{rem}
    The tow conditions of Theorems \ref{th3} and \ref{th4} are independent as it was shown in \cite[rem 3.7]{Bl-De-Mo}.
\end{rem}

\section*{acknowledgement}
    The authors would like to thank Professors Khalifa Dabbek and Jean-Pierre Demailly for many fruitful discussions concerning this paper.

\end{document}